\documentclass[12pt,reqno]{amsart}
\usepackage{blindtext}
\usepackage{comment}

\usepackage[OT2,T1]{fontenc}
\newcommand\textcyr[1]{{\fontencoding{OT2}\selectfont #1}}

\makeatletter

\usepackage[a4paper]{geometry}
\usepackage[dvipsnames]{xcolor}
\usepackage{mathtools,amssymb,bm,amsthm} 
\usepackage[dvipsnames]{xcolor} 
\usepackage{graphicx}
\usepackage{comment}
\usepackage{tikz-cd}
\usepackage[e]{esvect} 
\usepackage{breakcites} 
\usepackage{hyperref}
\usepackage{calrsfs}
\usepackage{hyperref} 
\hypersetup{
    colorlinks = true,
    linkcolor = {RedOrange},
    citecolor = {Green},
} 
\usepackage{caption} 
\usepackage{subcaption} 
\usepackage[cal=euler]{mathalfa}

\numberwithin{equation}{part}

\usepackage{titlesec}
\titleformat{\part}
{\normalfont\normalsize\bfseries\centering}{\thepart.}{1em}{} 

\titleformat{\section}
{\normalfont\normalsize\bfseries\centering}{\thepart.\thesection}{1em}{} 

\usepackage{chngcntr}
\counterwithin*{section}{part}

\usepackage{overpic}

\theoremstyle{plain}

\newtheorem{theorem}{Theorem}[part]

\newtheorem{lemma}[theorem]{Lemma}
\newtheorem{proposition}[theorem]{Proposition}
\theoremstyle{definition}
\newtheorem{definition}[theorem]{Definition}

\newtheorem{remark}[theorem]{Remark}

\newcommand{\ZZ}{\mathbb{Z}}

\newcommand{\Tr}{\mathrm{Trace}}

\title[Markov's conjecture on integral necklaces]{Markov's conjecture on integral necklaces}

\author{David Fisac}
\address{David Fisac, Departament de Matem\`atiques, Universitat Aut\`onoma de Barcelona, Bellaterra, Spain.}
\email{david.fisac@uab.cat}

\keywords{}
\subjclass{57K20, 11J06}

\thanks{The author is supported by the Luxembourg National Research Fund PRIDE17/1224660/GPS, the FEDER/AEI/MICINN grant PID2021-125625NB-I00 \& the AGAUR grant 2021-SGR-01015
}

\usepackage{setspace}

\usepackage{lipsum}

\begin{document}
\setstretch{1.15}

\begin{abstract}
 We use the geometric reformulation of Markov's uniqueness conjecture in terms of the simple length spectrum of the modular torus to rewrite the conjecture in combinatorial terms by explicitly describing this set of lengths.\vspace{-30pt}
\end{abstract}

\maketitle

\part{Result}
\vspace{-10pt}
Denote by $\mathcal{N}$ the set of primitive necklaces of positive integers with small variation (i.e. cyclic classes of finite sequences of positive integers such that the sum of any two blocks of the same size has difference at most $1$) and the trivial necklace $[0]$, see Definition \ref{def:smvar}. Primitivity refers to not being a power by concatenation of a smaller sequence.

Define the following function on this set of necklaces. 
\begin{align} \label{eq:Phi}
\begin{split}
\Phi:\mathcal{N}&\to \mathbb{Z}_{>0}\\
[n_1,\hdots, n_k]&\mapsto\frac{1}{10^k\cdot 2^{n_1+\cdots +n_k}}\sum_{S\subseteq\{1,\hdots,k\}}3^{r(S)-1}2^{k-r(S)}(\xi+2)^{|S|}(\overline\xi+2)^{|S^c|}\xi^{\sum_{i\in S}n_i}\overline\xi^{\sum_{i\in S^c}n_i},
\end{split}
\end{align}
for $\xi=3+\sqrt{5}$ and $\overline\xi =3-\sqrt{5}$, where $r(S)=\sum_{s\in\{\text{runs of }S\text{ and }S^c\}}|s|-1$ if $S, S^c\neq \{1,\hdots, k\}$, and $r(\{1,\hdots, k\})=r(\emptyset)=k$. A \emph{run} $s$ of $S\subseteq \{1,\hdots, k\}$ is a maximal subset $s\subseteq S$ such that the numbers are (cyclically) consecutive.

We prove an equivalent statement to Markov's uniqueness conjecture as follows. 

\begin{theorem}\label{mainthm} Markov's uniqueness conjecture is equivalent to 

\begin{center}\textit{The function $\Phi$ is injective in $\mathcal{N}$.}\end{center}
Moreover, the set of Markov numbers is exactly $\mathrm{Im}(\Phi)$.
\end{theorem}

Previous work in \cite{FL24} studies the strong rigidity of the set $\mathcal{N}$, meaning that it gives a parametrization of the set via choosing the consecutive integers appearing in a necklace and the number of occurrences of each, finding the bijection (see \cite[Proposition~2.8]{FL24}),
\[\mathcal{N}\leftrightarrow\{(x,y,m)\in \ZZ_{>0}^3\mid \gcd(x,y)=1\}\cup\{(1,0,m)\mid m\in\ZZ_{\geq0}\}.\] 

\begin{remark}
    Necklaces in $\mathcal{N}$ coincide with the cyclic classes of exponents of the most appearing term in finite Sturmian words or Christoffel words (see e.g. \cite[Chapter 7]{Aig13}), generated by the lower stair sequence of the segment in $\mathbb{R}^2$ joining $(0,0)\in \ZZ^2$ and $(m,n)\in \ZZ_{>0}^2$ for $\gcd(m,n)=1$. Section \ref{part:sv} will cover the background of the set $\mathcal{N}$.
\end{remark}

\section*{Acknowledgments}
I want to thank Mingkun Liu wholeheartedly for all the work we have done together before, which really shaped our understanding and made this possible. I’m also very grateful to Hugo Parlier, not just for the great discussions and feedback, but especially for suggesting we explore this idea in the first place. I would also like to thank Jonah Gaster for very appreciated comments and insights on this result and the bigger picture of the topic. Finally, I would like to thank the (anonymous) referee for such a thorough check, noticing the computational mistakes and making this work much more enjoyable to read.

\part{Introduction}

A triple of positive integers $x,y,z>0$ such that \[
x^2+y^2+z^2=3xyz
\]
is called a \emph{Markov triple} and were introduced by Markov in the late XIX century. Frobenius conjectured in \cite{Fro13} that these triples are always determined by their largest number, also called a \emph{Markov number}. This is known as \emph{Markov's uniqueness conjecture}, see \cite{Aig13}. Since then, many partial results have been made for Markov numbers of a given form, such as prime powers or particular linear functions of prime powers, see e.g. \cite{Bar96,Zha07}. However, the general statement remains open.

Define the modular torus $\mathcal{M}$ as the unique hyperbolic structure on the once-punctured torus such that the isometry group is of maximal order, it being of order $12$. As proved by Cohn in \cite{Coh71}, there is a map from Markov's triples to lengths of simple closed geodesics on the modular torus. This allows us to reformulate Markov's conjecture in geometric terms. 

Markov's conjecture is equivalent to the \emph{geometric Markov's uniqueness conjecture:} \textit{For every two simple closed geodesics on $\mathcal{M}$ of the same length, there is an isometry bringing one to the other.} We refer to \cite{MP08} for further explanation on this equivalence. This has a translation in terms of the multiplicity of the lengths of oriented simple closed geodesics on the modular torus: the first two lengths have multiplicity $6$ and all the rest have multiplicity $12$. Here multiplicity refers to the number of geodesics attaining these lengths. This is due to the shortest $12$ geodesics being invariant (up to orientation) by one of the isometries, whilst the rest are not. Moreover, one can compute explicitly their lengths, as it will be later done, getting to the expression of the spectrum in Remark \ref{rk:spec}.

Our main result is a reformulation of the geometric version of the conjecture via explicitly computing the lengths of all simple geodesics in combinatorial terms to find a new combinatorial version of the statement, with a new explicit description of the set of all Markov numbers. There is a result of a similar nature in \cite{GL21} where Gaster and Loustau use McShane's identity in \cite{McS98} on the simple closed geodesics of the modular torus to find a reformulation of the conjecture as an identity on Lagrange numbers. 

\part{Small-variation necklaces}\label{part:sv}

Necklaces with small variation will appear in this work as the encoding of the set of simple closed curves on the once-punctured torus in Theorem \ref{thm:bus} from \cite{BS88}, giving finally the domain on which Markov numbers are described in Theorem \ref{mainthm}. The following contains only background on them. Readers already familiar with the topic might skip directly to Section \ref{part:dev}. Let us write their definition.

\begin{definition}[Necklace with small variation]\label{def:smvar} A \emph{necklace of integers} is a class of finite sequences of integers with respect to cyclic shifting. It is said to have \emph{small variation} if every two (cyclic) blocks of the same size have sum difference at most $1$, i.e. $[n_1,\hdots,n_k]$ with $n_i\in\ZZ$ has small variation if for any $s\geq 0$, and for any $i_0,j_0\in\{1,\hdots k\}$,
\begin{equation*} \label{eq:smaVa}
    \bigg| \sum_{i=i_0}^{i_0+s} n_{i} - \sum_{j=j_0}^{j_0+s} n_{j} \bigg| \leq 1,
\end{equation*}
    where indices are taken modulo $k$. We will call such a necklace to be \emph{primitive} whenever it cannot be represented by a power by concatenation of a smaller sequence.
\end{definition}

The small-variation condition for blocks of size $1$ implies directly that any small-variation necklace contains at most two integers and they have to be consecutive. We will follow this section tracking the example of the necklace $[3,4,4,3,4,3,4,4,3,4,3,4]$.

The small-variation condition applies to a necklace as an almost equidistribution of the configuration of the two integers appearing in the necklace. This is an analog to the balanced condition (see \cite{Vui03}) for the cyclic words in $a,b$ with all exponents of $a$ being $1$ and exponents on $b$ following a small-variation necklace, i.e., in our example, the cyclic word $[ab^3ab^4ab^4ab^3ab^4ab^3ab^4ab^4ab^3ab^4ab^3ab^4]$.

\begin{figure}[h!]
\centering
\includegraphics[width=0.5\linewidth]{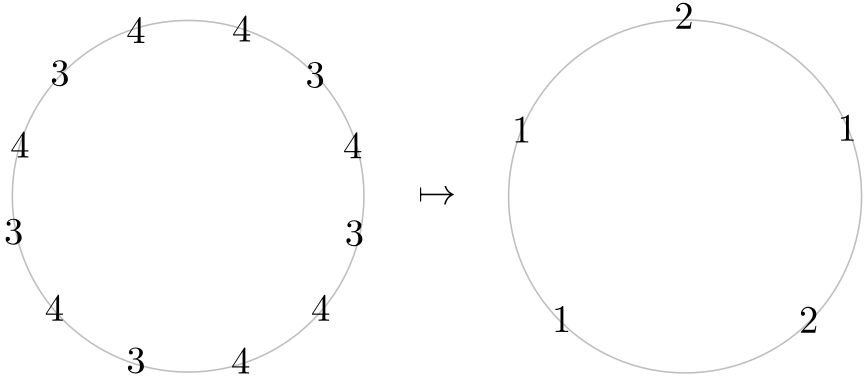}
\caption{}
    \label{fig:red}
\end{figure}

In a small-variation necklace, the integer that appears the least frequently will always appear isolated, whilst the other one will appear in larger runs, i.e. maximal cyclic subsets with respect to inclusion with all entries being the same. This allows for a reduction of the necklaces via mapping every small-variation necklace to the necklace consisting of the run sizes of the
integer appearing most frequently, as in Figure \ref{fig:red}. This reduction is proven to preserve the small-variation condition in \cite{FL24} and can be reversed with the extra information of the two consecutive integers appearing and which one appears most frequently, see also \cite{Kon85}.

A consequence of this reduction is that given the two consecutive integers $m,m+1\in\mathbb{Z}_{>0}$ appearing and the number of occurrences of each, let us call them $x,y\in\mathbb{Z}_{\geq0}$, respectively, there is a unique small-variation necklace with these features, see \cite[Proposition~2.8]{FL24}. Moreover, this necklace is primitive if and only if $\gcd(x,y)=1$. In our example these correspond to $m=3$, $x=5$ and $y=7$. The above gives a bijection of the set $\mathcal{N}$ appearing in Theorem \ref{mainthm} as follows
\[\mathcal{N}\leftrightarrow\{(x,y,m)\in \ZZ_{>0}^3\mid \gcd(x,y)=1\}\cup\{(1,0,m)\mid m\in\ZZ_{\geq0}\}.\]

Another approach to the above bijection is by understanding the relation between necklaces and Christoffel words, see \cite{Aig13} for more background on these objects. The following are two posible approaches, represented in Figure \ref{fig:christ} for our example. Fix a given primitive small-variation necklace composed of integers $m$ and $m+1$, appearing $x$ and $y$ times respectively. First, draw the segment in $\mathbb{R}^2$ joining $(0,0)$ and $(x,y)$, and consider the highest stair on the integer grid bellow the segment. By writing $m$ for every unit of horizontal segment and $m+1$ for every unit of vertical segment, the sequence appearing represents the desired small-variation necklace. Alternatively, draw the segment joining $(0,0)$ and $(x+y,xm+y(m+1))$. Consider again the highest stair on the integer grid bellow the segment. By writing $a$ for every unit of horizontal segment and $b$ for every unit of vertical segment, we find a word of the form $ab^{n_1}\cdots ab^{n_k}$ where $[n_1,\hdots, n_k]$ is our original necklace. In our example, the stairs in Figure \ref{fig:christ} correspond to the segments associated to $(5,7)$ and to $(12,43)$, in both descriptions above.

\begin{figure}[h!]
    \centering
    \begin{overpic}[width=0.9
    \linewidth]{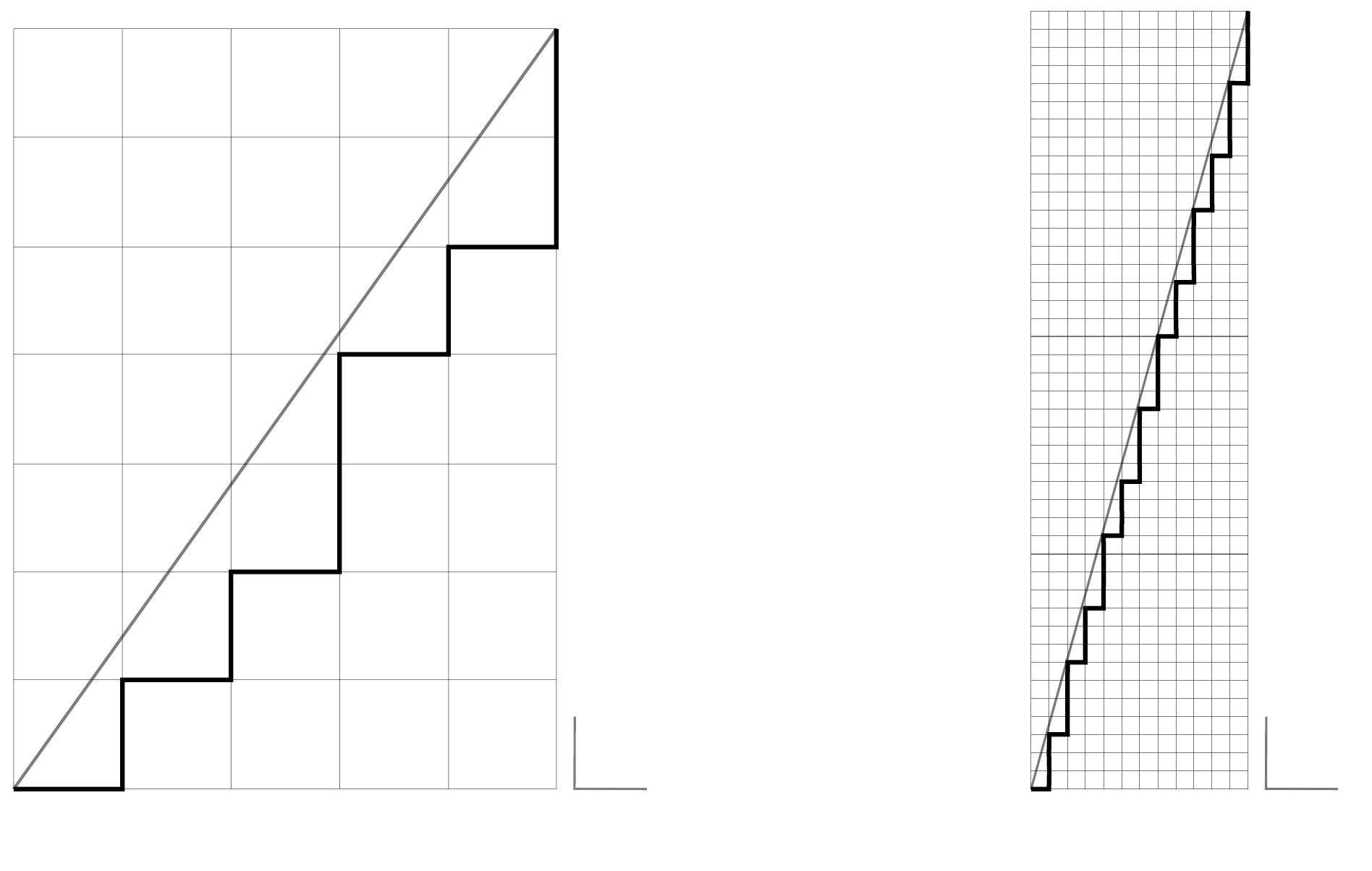}
    \put(47.5,5){$3$}
    \put(41,12){$4$}
    \put(98,5.25){$a$}
    \put(91.5,12){$b$}
    \put(5,0){$[3,4,3,4,3,4,4,3,4,3,4,4]$}
    \put(60,0){$[ab^3ab^4ab^3ab^4ab^3ab^4ab^4ab^3ab^4ab^3ab^4ab^4]$}
    \end{overpic}
    \caption{}
    \label{fig:christ}
\end{figure}

Once the appearing integers and the numbers of appearences are prescribed, the combinatorics of small-variation necklaces rely on understanding what is the unique configuration of these numbers that satisfies the small variation condition, i.e. how do they distribute. The above are different descriptions of the configurations that these necklaces naturally realize. Note that understanding the different characterizations of these configurations plays a big role in their image by the function $\Phi$ in Theorem \ref{mainthm}, given that the run function $r$ gives different weights to subsets of the necklace depending on how many adjacent numbers they have. The run function is the only term depending on how the numbers are distributed and separates small-variation necklaces from other necklaces with the same elements but different configuration.

\part{Development}\label{part:dev}
Let $S$ be a topological once-punctured torus, i.e. a surface of signature $(1,1)$. Fix $\{a,b\}$ canonical generators of $\pi_1(S)\cong \mathbb{F}_{a,b}$ as in Figure \ref{fig:gen}. Define (oriented) closed curves as free homotopy classes of loops on $S$, that are in correspondence with conjugacy classes of $\pi_1(S)$. Define the set of closed curves as $\mathcal{C}(S)$. These are thus in correspondence with cyclic shift classes of cyclically reduced words in $\{a,b,a^{-1},b^{-1}\}$. Fix an ideal triangulation in $S$ as the colored one in Figure \ref{fig:gen}. Consider its dual graph $\Gamma$, a trivalent graph with two vertices embedded in $S$, also pictured in Figure \ref{fig:gen}.

Recall that the Teichmüller space of $S$, denoted by $\mathcal{T}(S)$ is the space of marked hyperbolic structures on the once-punctured torus $S$. As was first introduced by Thurston in \cite{Thu98}, it can be parametrized by shear coordinates in the fixed ideal triangulation, measuring the hyperbolic displacement between adjacent ideal triangles. The modular torus is defined as the hyperbolic structure on $S$ with zero shear coordinates and will be denoted by $\mathcal{M}$. 

\begin{figure}[h!]
    \centering
    \begin{overpic}[width=0.5\linewidth]{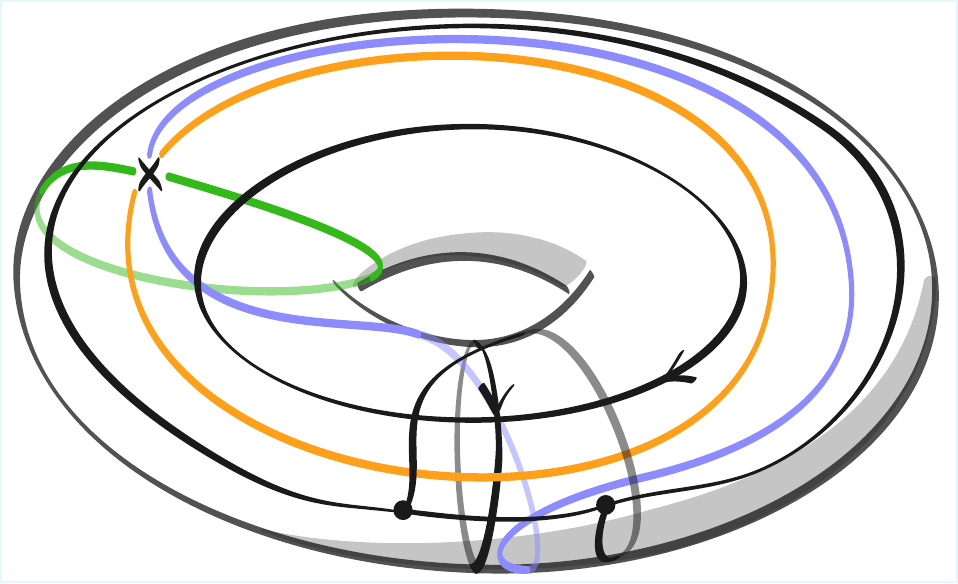}
    \put(3,50){$S$}
    \put(37,3.5){$\Gamma$}
    \put(69,24){$b$}
    \put(53,21){$a$}
    \end{overpic}
    \caption{}
    \label{fig:gen}
\end{figure}

Given the dual graph of the fixed triangulation, there is a natural map 
\[
\omega: \mathcal{C}(S)\to \mathrm{words}(L,R)/\,\mathrm{cyclic\,\, shift}
\]
by the following procedure. Given a closed curve, retract it to $\Gamma$ without backtracking. Fix a startpoint. This gives a finite sequence of Left/Right turns on the graph until reaching the startpoint again. Consider this sequence as a word in "L", and "R". Forgetting the startpoint corresponds to considering cyclic shifts of the words. 

Define two matrices, 
\[L=\begin{pmatrix}
1 & 1 \\
0 & 1 
\end{pmatrix}\text{, and }R=\begin{pmatrix}
1 & 0 \\
1 & 1 
\end{pmatrix}.\]

Let $\ell_\mathcal{M}:\mathcal{C}(S)\to \mathbb{R}_{\geq0}$ denote the hyperbolic length of a curve in the modular torus via measuring the length of the unique geodesic in the homotopy class. Identifying composition of words with multiplication of matrices, the following trace formula holds for the modular torus, see e.g. \cite[p.~8]{MS16},
\[
\ell_\mathcal{M}(\gamma)= 2\cdot \mathrm{acosh}(\mathrm{Trace}(\omega(\gamma))/2).
\]

Recall that a simple curve is a curve without self-intersection, i.e., that it can be represented by a loop on $S$ that does not intersect itself. We want to study the simple length spectrum of $\mathcal{M}$, $\mathcal{SS}(\mathcal{M})$, i.e., the set of lengths of oriented simple closed curves with respect to $\ell_\mathcal{M}$, with multiplicity. There is a well-known classification of all words in $\{a,b,a^{-1},b^{-1}\}$ corresponding to simple closed curves.

\begin{theorem}[{\cite[Theorem~6.2]{BS88}}] \label{thm:bus}
    Every simple closed curve on $S$ can be represented, after suitably renaming the generators in $\{a,b,a^{-1},b^{-1}\}$, by one of the following words:
    \begin{enumerate}
        \item \label{type1}
            $a$,

        \item \label{type2}
            $aba^{-1}b^{-1}$,

        \item \label{type}
            $ab^{n_1} a b^{n_2} \cdots a b^{n_r}$, where $[n_1, \dots, n_r]$ is a necklace of positive integers with small variation (see Definition \ref{def:smvar}).
    \end{enumerate}
    Conversely, each of these words is homotopic to a power of a simple closed curve.
\end{theorem}

We want to study the lengths of the curves in Theorem \ref{thm:bus} with the hyperbolic structure $\mathcal{M}$. The curves of the second type correspond to simple curves bounding the puncture and have zero hyperbolic length, hence we are only interested in the first and third types. Curves of the first type are only $4$, which are $a,b,a^{-1},b^{-1}$, which all correspond to the cyclic word $LR$ in the graph $\Gamma$. 

For the curves of the third type, note that for a given necklace of positive integers $[n_1,\hdots, n_k]$ with small variation different from $[1]$, there are eight closed curves via renaming the generators. These are \[ab^{n_1}\cdots ab^{n_k}, ba^{n_1}\cdots ba^{n_k}, ab^{-n_1}\cdots ab^{-n_k}\text{, and }ba^{-n_1}\cdots ba^{-n_k},\] and their inverses. These correspond to the cyclic words in $\Gamma$ \[R(RL)^{n_1-1}L\cdots R(RL)^{n_k-1}L, L(LR)^{n_1-1}R\cdots L(LR)^{n_k-1}R,\] \[ L(LR)^{n_1}R\cdots L(LR)^{n_k}R\text{, and },R(RL)^{n_1}L\cdots R(RL)^{n_k}L\] respectively. Moreover, the first two will have the same trace and so will the last two. Similarly, to the exceptional necklace $[1]$ there are only four closed curves via renaming the generators, which are $ab$, $ab^{-1}$ and their inverses, corresponding to the cyclic words in the graph $\Gamma$: $LR$ and $LLRR$, respectively.

 Let $\omega=[n_1,\hdots, n_k]$ be a necklace of positive integers. Asssume $n_i\in\{m,m+1\}$ for some $m\in \ZZ_{>0}$ for all $i=1,\hdots, k$, as otherwise the necklace cannot have small variation. Let $|\omega|_m$ and $|\omega|_{m+1}$ be the number of occurrences of $m$ and $m+1$. Denote by $\mathcal{N}$ the \emph{set of primitive necklaces of positive integers with small variation} and the trivial necklace $[0]$. Denote by $\mathcal{N}'=\mathcal{N}\setminus\{[0],[1]\}$. By \cite[Proposition~2.8]{FL24}, given $m,x,y\in \ZZ_{>0}$, there is a unique necklace with small variation $[n_1,\hdots,n_k]$ with $n_i\in \{m,m+1\}$ for all $i=1,\dots k$ and with $|\omega|_m=x$, $|\omega|_{m+1}=y$. Moreover, the necklace is primitive (not represented by a power of a shorter sequence) if and only if $\gcd(x,y)=1$.

Thus, defining the real-valued function \begin{equation}\label{eq:theta}\theta(\_)=2\cdot \mathrm{acosh}(\mathrm{Trace}(\_)/2)\end{equation} on $SL_2(\ZZ)$, we find, recalling that $R = L^T$,
\begin{align*}
\mathcal{SS}(\mathcal{M})=6\{\theta(LR)\}\,&\cup\,2\{\theta(LLRR)\}\,\cup\,4\{\theta(L(LR)^{n_1}R\cdots L(LR)^{n_k}R)\mid [n_1,\hdots, n_k]\in\mathcal{N}'\}\\&\cup\,4\{\theta(L(LR)^{n_1-1}R\cdots L(LR)^{n_k-1}R)\mid [n_1,\hdots, n_k]\in\mathcal{N}'\},
\end{align*}
where the scalar multiples denote multiplicity of the set.

The small variation condition is invariant under the choice of the two consecutive integers appearing. Define $\mathcal{N}^*$ as the set of primitive necklaces with all elements in $\{0,1\}$ with small variation, excluding the two simplest necklaces $[0]$ and $[1]$. It follows that 
\begin{align}\label{eq:spec}
\begin{split}
\mathcal{SS}(\mathcal{M})=&8\{\theta(L(LR)^{n_1}R\cdots L(LR)^{n_k}R)\mid [n_1,\hdots, n_k]\in\mathcal{N}'\}\,\cup\\ &4\{\theta(L(LR)^{n_1}R\cdots L(LR)^{n_k}R)\mid [n_1,\hdots, n_k]\in\mathcal{N}^*\}\,\cup\\ &6\{\theta(LR)\}\,\cup\, 6\{\theta(LLRR)\}.
\end{split}
\end{align}

\begin{lemma}\label{lem:eta} Let $\eta=[\eta_1,\hdots, \eta_l]\in \mathcal{N}^*$. Write it uniquely (up to cyclic shifting) as \[\eta=[\underbrace{0,\hdots,0}_{n_1-1},1,\underbrace{0,\hdots,0}_{n_2-1},1,\hdots,\underbrace{0,\hdots,0}_{n_k-1},1].\]
Then, 
\[
\Tr(L(LR)^{\eta_1}R\cdots L(LR)^{\eta_l}R)=\Tr(L(LR)^{n_1}R\cdots L(LR)^{n_k}R).
\]
\end{lemma}
\begin{proof}
    The lemma follows from elementary properties of the trace. We will write it for convenience of the reader. First of all, 
    \[
    L(LR)^{\eta_1}R\cdots L(LR)^{\eta_l}R=(LR)^{n_1-1}LLRR \cdot(LR)^{n_2-1}LLRR\cdot \,\cdots\,\cdot  (LR)^{n_k-1}LLRR.  
    \]
    Applying a double cyclic shift and regrouping one finds 
    \begin{align*}RR(LR)^{n_1-1}LLRR \cdot(LR)^{n_2-1}LLRR\cdot \,\cdots\,\cdot  &(LR)^{n_k-1}LL = \\  &=R(RL)^{n_1}L\cdot R(RL)^{n_2}L\cdots R(RL)^{n_k}L,\end{align*}
    which has the same trace as $L(LR)^{n_1}R\cdots L(LR)^{n_k}R$ since $R=L^T$.
\end{proof}
Moreover, note that in the bijection
\begin{align*}
    \Theta:\{\text{necklace of positive integers}\}&\longrightarrow \{\text{necklace of integers in }\{0,1\}\}\\
    \eta=[n_1,\hdots, n_k]&\longmapsto [\underbrace{0,\hdots,0}_{n_1-1},1,\underbrace{0,\hdots,0}_{n_2-1},1,\hdots,\underbrace{0,\hdots,0}_{n_k-1},1],
\end{align*}
it follows from \cite[Lemma~2.9]{FL24} that a necklace of positive integers $\eta$ has small variation if and only if $\Theta(\eta)$ has, as $\eta$ is the necklace of runs of its image plus $1$, ensuring positivity.
Hence, it follows from Lemma \ref{lem:eta} that
\[
\mathcal{SS}(\mathcal{M})=12\{\theta(L(LR)^{n_1}R\cdots L(LR)^{n_k}R)\mid [n_1,\hdots, n_k]\in\mathcal{N}'\}\,\cup 6\{\theta(LR)\}\,\cup\, 6\{\theta(LLRR)\},
\] 
and so the geometric Markov's uniqueness conjecture is true if and only if 
\[
|\{[n_1,\dots, n_k]\in\mathcal{N}' \mid \mathrm{Trace}(L(LR)^{n_1}R \cdots L(LR)^{n_k}R)=\ell \}|\in \{0,1\} \text{ for all }\ell\in\mathbb{R}.
\]

The last ingredient to prove is the rewriting of this trace in combinatorial terms of the simple length spectrum.
\begin{proposition} Let $\xi=3+\sqrt{5}$ and $\overline\xi=3-\sqrt{5}$. For any $n_1,\dots n_k\in \mathbb{Z}_{\geq0}$, 
\begin{align*}
\mathrm{Trace}(L(LR)^{n_1}&R \cdots L(LR)^{n_k}R)=\\&=\frac{1}{10^k\cdot 2^{n_1+\cdots +n_k}}\sum_{S\subseteq\{1,\hdots,k\}}3^{r(s)}2^{k-r(S)}(\xi+2)^{|S|}(\overline\xi+2)^{|S^c|}\xi^{\sum_{i\in S}n_i}\overline\xi^{\sum_{i\in S^c}n_i},
\end{align*}
where $r(S)=\sum_{s\in\{\text{runs of }S\text{ and }S^c\}}|s|-1$ if $S\neq \{1,\hdots, k\}$ and $S\neq\emptyset$, and $r(\{1,\hdots, k\})=r(\emptyset)=k$.
    
\end{proposition}

\begin{proof}
Note that
\[
LR=\begin{pmatrix}
2 & 1 \\
1 & 1 
\end{pmatrix}=P\cdot \frac{1}{2}\begin{pmatrix}
\overline{\xi} & 0 \\
0 & \xi 
\end{pmatrix}\cdot P^{-1},
\]
where $\xi=3+\sqrt{5}$, $\overline{\xi}=3-\sqrt{5}$, $P=\frac{1}{2}\begin{pmatrix}
    \overline{\xi}-2 & \xi -2 \\ 2 & 2
\end{pmatrix}$ and $P^{-1}=\frac{1}{10}\begin{pmatrix}
    2(\overline{\xi}-3) & \xi+2 \\ 2(\xi-3) & \overline{\xi}+2
\end{pmatrix}$.

Hence, 

\[
L(LR)^nR=\frac{1}{2^n}\cdot \tilde{P}\cdot  \begin{pmatrix}
\overline{\xi}^n & 0 \\
0 & \xi^n 
\end{pmatrix} \cdot \tilde{\tilde{P}},
\]
where $\tilde{P}=\frac{1}{2}\begin{pmatrix}
    \overline{\xi} & \xi \\ 2 & 2
\end{pmatrix}$ and $\tilde{\tilde{P}}=\frac{1}{10}\begin{pmatrix}
    \overline{\xi}+2 & \xi+2 \\ \xi+2 & \overline{\xi}+2
\end{pmatrix}$.

Thus,
\begin{align*}
\Tr(L(LR)^{n_1}&R\cdots L(LR)^{n_k}R)=\\ &=\frac{1}{10^{k}\cdot 2^{n_1+\cdots + n_k}}\cdot \Tr\Bigg(M\cdot \begin{pmatrix}
    \overline{\xi}^{n_1} & 0 \\ 0 & \xi^{n_1} 
\end{pmatrix}\cdot \, \cdots\, \cdot M\cdot \begin{pmatrix}
    \overline{\xi}^{n_k} & 0 \\ 0 & \xi^{n_k} 
\end{pmatrix} \Bigg),
\end{align*}
where $M=10\cdot \tilde{\tilde{P}}\cdot \tilde{P}=\begin{pmatrix}
    3(\overline{\xi}+2) & 2(\xi+2)\\ 2(\overline{\xi}+2) & 3(\xi+2)
\end{pmatrix}.$ 

Therefore, it is left to study 
\[
\Tr\Bigg(M\cdot \begin{pmatrix}
    \overline{\xi}^{n_1} & 0 \\ 0 & \xi^{n_1} 
\end{pmatrix}\cdot \, \cdots\, \cdot M\cdot \begin{pmatrix}
    \overline{\xi}^{n_k} & 0 \\ 0 & \xi^{n_k} 
\end{pmatrix} \Bigg)=\Tr(\prod_{l=1}^k A_l),
\]
with $A_l=\begin{pmatrix}
    3(\overline{\xi}+2)\overline{\xi}^{n_l} & 2(\xi+2)\xi^{n_l} \\ 2(\overline{\xi}+2)\overline{\xi}^{n_l} & 3(\xi+2)\xi^{n_l} 
\end{pmatrix}=(a_{i_1,i_2}^l)_{i_1,i_2}$.

Write $a_{i_1,i_2}^l=c_{i_1,i_2}\cdot d_{i_2}^l$ for $i_1,i_2=1,2$, and $l=1,\hdots k$,  where 
\[c_{i_1,i_2}=\begin{cases}
    3 \text{ if }i_1=i_2\\
    2 \text{ if }i_1\neq i_2
\end{cases}, \text{ and } d_{i_2}^l=\begin{cases}
    (\overline\xi+2)\overline\xi^{n_l}\text{ if }i_2=1 \\
    (\xi+2)\xi^{n_l}\text{ if }i_2=2.
\end{cases}
\]

Thus, 
\[
\Tr(\prod_{l=1}^k A_l)=\sum_{j_1=1}^2\cdots \sum_{j_k=1}^2 a_{j_1,j_2}^1\cdots a_{j_k,j_1}^k=\sum_{S\subseteq \{1,\hdots,k\}}a_{j_1,j_2}^1\cdots a_{j_k,j_1}^k, 
\]
where $j_i=1$ if $i\in S$ and $j_i=2$ otherwise. 

Hence, 
\begin{align*}
\Tr(\prod_{l=1}^k A_l)&=\sum_{S\subseteq \{1,\hdots,k\}}c_{j_1,j_2}\cdots c_{j_k,j_1}\cdot d_{j_2}^1\cdots d_{j_1}^k\\
&=\sum_{S\subseteq\{1,\hdots,k\}}3^{r(s)}2^{k-r(S)}(\xi+2)^{|S^c|}(\overline\xi+2)^{|S|}\xi^{\sum_{i\in S^c}n_{i-1}}\overline\xi^{\sum_{i\in S}n_{i-1}}\\
&=\sum_{S\subseteq\{1,\hdots,k\}}3^{r(s)}2^{k-r(S)}(\xi+2)^{|S|}(\overline\xi+2)^{|S^c|}\xi^{\sum_{i\in S}n_i}\overline\xi^{\sum_{i\in S^c}n_i},
\end{align*}
where the final equality follows by applying the permutation of $\mathcal{P}(\{1,\hdots,k\})$ that sends $S\subseteq \{1,\hdots,k\}$ to the
set $\{i - 1 : i \in S^c\}$. Note that this permutation preserves $r(S)$. This finishes the proof.
\end{proof}

\begin{remark} \label{rk:spec} Note that, by the above argument, we have proven the explicit description of the simple length spectrum of the modular torus $\mathcal{M}$ in terms of the function $\Phi$ in Theorem \ref{mainthm}. That is, 
\[
\mathcal{SS}(\mathcal{M})=12\{2\,\mathrm{acosh}(3\Phi(\omega)/2))\mid \omega \in \mathcal{N}'\}\cup 6\{2\,\mathrm{acosh}(3\Phi([1])/2))\}\cup 6\{2\mathrm{acosh}(3\Phi([0])/2))\}, 
\]
where again scalar multiples denote multiplicity of the set.    
\end{remark}

The second part of Theorem \ref{mainthm} follows from the work in \cite{Coh71}, who proved that for every Markov number $x\in\mathbb{Z}$, there is a simple closed geodesic $\gamma$ on $\mathcal{M}$ such that \[3x=2\mathrm{cosh}(\ell_\mathcal{M}(\gamma)/2).\]

\end{document}